\documentclass[12pt]{article}
\usepackage[top = 30truemm,bottom = 30truemm,left = 25truemm,right = 25truemm]{geometry}
\usepackage{amsmath,amssymb,amsthm,graphicx,mathrsfs,mathtools,enumerate}
\usepackage{tikz}
\usepackage[all]{xy}

\usepackage{hyperref}
\usepackage{xcolor}
\usepackage[capitalize,nameinlink,noabbrev,nosort]{cleveref}
\hypersetup{
	colorlinks=true,       
	linkcolor=black,          
	citecolor=black,        
	filecolor=black,      
	urlcolor=black,           
}

\newtheorem{theoremcounter}{Theorem Counter}[section]

\theoremstyle{definition}
\newtheorem{definition}[theoremcounter]{Definition}

\theoremstyle{remark}
\newtheorem{remark}[theoremcounter]{Remark}

\theoremstyle{plain}
\newtheorem{lemma}[theoremcounter]{Lemma}
\newtheorem{proposition}[theoremcounter]{Proposition}

\newtheorem{theorem}[theoremcounter]{Theorem}

\numberwithin{equation}{section}

\newcommand{\A}{\mathcal{A}}
\newcommand{\Z}{\mathbb{Z}}
\newcommand{\Q}{\mathbb{Q}}
\newcommand{\R}{\mathbb{R}}

\newcommand{\bk}{\mathbf{k}}

\newcommand{\ZZ}{\mathcal{Z}}

\newcommand{\zA}{\zeta^{}_\A}

\newcommand{\gaW}{\gamma_\A^\mathrm{W}}
\newcommand{\gaM}{\gamma_\A^\mathrm{M}}
\def\gaK#1{\gamma_\A^{\mathrm{K},{#1}}}

\def\st#1#2{\genfrac{[}{]}{0pt}{}{#1}{#2}}
\def\sts#1#2{\genfrac{\{}{\}}{0pt}{}{#1}{#2}}

\def\la{\ell_\A}

\def\={\,=\,}

\title{On finite analogues of Euler's constant}
\date{}
\author{Masanobu Kaneko, Toshiki Matsusaka, and Shin-ichiro Seki}

\begin{document}
\maketitle

\begin{abstract}
We introduce and study finite analogues of Euler's constant in the same setting as finite multiple zeta values.
We define a couple of candidate values from the perspectives of a ``regularized value of $\zeta(1)$'' and of Mascheroni's and Kluyver's series expressions of Euler's constant using Gregory coefficients.
Moreover, we reveal that the differences between them always lie in the $\Q$-vector space spanned by 1 and values of a finite analogue of logarithm at positive integers.
\end{abstract}

\section{Introduction}
Euler's constant $\gamma$, introduced by Euler~\cite{Eul1} as the limit
\[ \gamma=\lim_{n\to\infty} \left(1+\frac12+\frac13+\cdots+\frac1n-\log n\right), \]
is one of the most renowned, mysterious mathematical constants. In particular, although there is
a vast amount of works in the literature on $\gamma$, it is still unknown if $\gamma$ is an irrational number or not. See Lagarias's survey paper~\cite{Lagarias2013} for more details
and in particular Euler's work on his constant $\gamma$. 

In this article, instead of investigating the real number $\gamma$ itself, we define and study several ``finite" analogues of Euler's constant in the setting 
which has been drawn much attention in recent years in connection to the {\it finite
multiple zeta value}. To motivate our study, we first briefly recall this rather new object and a notable conjecture concerning 
a connection to its real counterpart.

Let $\A$ be the quotient ring 
\[\A\coloneqq\left. \left(\prod_p \Z/p\Z\right) \middle/ \left(\bigoplus_p \Z/p\Z\right) \right.\]
of the direct product of $\Z/p\Z$ over all primes modulo the ideal of the direct sum.
The field $\Q$ of rational numbers can naturally be embedded diagonally in $\A$,
and via this the ring $\A$ is regarded as a $\Q$-algebra.

For each tuple of positive integers $(k_1,\ldots, k_r)$, we consider
a finite multiple zeta value $\zA(k_1,\ldots, k_r)$ in $\A$ given by\footnote{We often identify
an element in $\A$ with its representative in $\prod_p \Z/p\Z\simeq\prod_p \Z_{(p)}/p\Z_{(p)}$, where $\Z_{(p)}$ is the localization of $\Z$ at $(p)$.} 
\[ \zA(k_1,\ldots, k_r)\coloneqq\left(\sum_{0<m_1<\cdots <m_r<p}\frac1{m_1^{k_1}\cdots m_r^{k_r}}\ \bmod p\right)_p \in\A. \]
The $\Q$-vector space $\ZZ_\A$ spanned by all finite multiple zeta values becomes a $\Q$-algebra
by the so-called stuffle product structure. Then, our fundamental conjecture is the existence of an isomorphism of
$\Q$-algebras between $\ZZ_\A$ and $\ZZ_\R/\zeta(2)\ZZ_\R$, where $\ZZ_\R$ is the $\Q$-algebra of usual multiple 
zeta values in $\R$ and $\zeta(2)\ZZ_\R$ is the ideal of $\ZZ_\R$ generated by $\zeta(2)$ ($\zeta(s)$ is the Riemann zeta function):
\[    
\begin{array}{ccc}
 \ZZ_\A & \stackrel{?}{\simeq} & \ZZ_\R/\zeta(2)\ZZ_\R \\
\rotatebox{90}{$\in$} & & \rotatebox{90}{$\in$} \\
\zA(\bk) & \longleftrightarrow & \zeta^{}_\mathcal{S}(\bk) \end{array}
\]
with explicitly defined element $\zeta^{}_\mathcal{S}(\bk)$ in $\ZZ_\R/\zeta(2)\ZZ_\R$
which conjecturally corresponds to $\zA(\bk)$ under the predicted isomorphism. See~\cite{Kan19, KZ} for more details, and~\cite{Zhao}
for multiple zeta values in general.

Let $B_n$ be the Bernoulli number defined by the generating series
\[ \frac{x}{1-e^{-x}}=\sum_{n=0}^\infty B_n\frac{x^n}{n!} \]
(we adopt the definition with $B_1=\frac12$) and consider for each $k\ge2$ an element $Z_\A(k)$ in $\A$ given by 
\[ Z_\A(k)\coloneqq\left(\frac{B_{p-k}}k\,\bmod p\right)_p.\]
Note that by definition we may ignore finitely many $p$-components (in this case $p$'s with $p<k$) in order to 
give an element of $\A$. 

There are overwhelming evidences that the element in $\A$ which corresponds to
$\zeta(k)\bmod \zeta(2)\ZZ_\R$ under the conjectural isomorphism above should be $Z_\A(k)$. See, for instance, \cite{SW}, \cite{Murahara2016} for an $\A$-version of the so-called
sum formula, where the correspondence $Z_\A(k)\leftrightarrow\zeta(k)$ is visibly conspicuous, and \cite{KOS}, \cite{FujitaKomori2021} for similar evidence in the Aoki--Ohno relation.

A heuristic argument to ``justify'' this correspondence is 
\[ \zeta(k)\;\,\text{``}\!\underset{\text{Fermat}}{\equiv}\text{''}\;\zeta(k-(p-1))
\underset{\text{Euler}}{\=}-\frac{B_{p-k}}{p-k}
\equiv \frac{B_{p-k}}k\ \pmod p.\]
``Fermat'' means forcibly applying Fermat's little theorem to each summand of the Riemann zeta value. This can also be interpreted as a forbidden extrapolation of Kummer's congruence, which are valid for negative integers, to positive integers.

The aim of the present article is, from the perspectives of $\gamma$ as a regularized constant for $\zeta(1)$ and as series involving Gregory coefficients, to introduce possible analogues of $\gamma$ in $\A$ in essentially two ways, 
and discuss their interrelation and some related matters.

\section{Two analogues of $\gamma$, Wilson and Mascheroni}

We proceed with a similar heuristic as for $\zeta(k)$ and define an analogue $\gaW$ of Euler's constant in $\A$ as follows.
Let 
\[ W_p\coloneqq\frac{(p-1)!+1}p\in\Z \]
be the Wilson quotient. Throughout the article, $p$ always denotes a prime number.

\begin{definition}  We define $\gaW$ in $\A$ by
\[ \gaW\coloneqq\left(W_p\bmod p\right)_p\in\A. \]
\end{definition}
We do not know whether $\gaW$ is zero or not, but if there exist infinitely many non-Wilson primes (primes such that $W_p$ is not divisible by $p$), we can conclude that $\gaW$ is non-zero.

A heuristic explanation to argue that this is an analogue of Euler's constant is this\footnote{The first named author learned about ten years ago from Don Zagier that Maxim Kontsevich (in a private communication), who has considered the ring $\A$ in a different context~\cite{Kon}, considered this $\gaW$ should be an analogue of $\gamma$ in $\A$. His heuristic argument was, so it seemed, considerably different from the current one. Later, the third named author learned from Go Yamashita that one might regard the Wilson quotients as an analogue of Euler's constant in $\A$. His argument is different from both Kontsevich's and the one discussed in this paper.}. 
First, recall that $\gamma$ is a ``regularized value of $\zeta(1)$'', or, 
\[ \gamma=\lim_{s\to1} \left(\zeta(s)-\frac1{s-1}\right).\]
As before, we apply the reasoning 
\[ \zeta(1)\,\text{``}\equiv\text{"}\,\zeta(1-(p-1))=-\frac{B_{p-1}}{p-1}, \]
ending up in the value which is not $p$-integral. To remedy this, we add a ``polar term'' and have
\[ -\frac{B_{p-1}}{p-1}+\frac1p\in\Z_{(p)}, \]
by the theorem of Clausen and von Staudt~\cite[Theorem~3.1]{AIK}. If we define $\beta_n\ (n\ge1)$ by
\[ \beta_n\coloneqq\begin{cases}B_n/n & \text{if }p-1\nmid n, \\ (B_n+p^{-1}-1)/n & \text{if } p-1\mid n,\end{cases} \]
the value above is $-\beta_{p-1}$, and notably, it is known that various modulo $p$ congruences for $\beta_n$ hold uniformly, regardless of whether $n$ is divisible by $p-1$ or not (cf.~\cite{Johnson1975}). In particular, Kummer's congruence $\beta_{2n}\equiv \beta_{2n+p-1}$ ($p\geq 5$) holds for any $n$ and, by the (forbidden) extrapolation, the ``congruence''
\[ \text{``regularized value of }\zeta(1)\text{'' ``}\equiv\text{''} -\beta_{p-1}=-\frac{1}{p-1}\left(B_{p-1}+\frac1p-1\right) \]
might look natural.
Here we note that it has been known since more than a century that the last quantity in the parentheses is congruent to the Wilson quotient $W_p$ modulo $p$:
\begin{proposition}[Glaisher~\cite{Glaisher1900}]
	For any odd prime $p$,
	\[
		W_p \equiv B_{p-1} + \frac{1}{p} -1 \pmod{p}.
	\]
\end{proposition}
One way to prove this is by using an explicit formula for Bernoulli numbers in terms of Stirling numbers of the second kind
(\cite[Theorem~2.8]{AIK}, see~\eqref{eq:BernStir} below) and a congruence for Stirling numbers.
  
In this way, we may regard $\gaW$ as an analogue of $\gamma$ in $\A$.
Further, we introduce yet another analogue of Euler's constant in $\A$ and establish a relation with $\gaW$ and other quantities in $\A$ such as an analogue $\log_\A(x)$ of logarithm.

Let $(G_n)_n$ be Gregory coefficients (also known as Bernoulli numbers of the second kind) defined by
\begin{equation}\label{eq:gregory} \frac{x}{\log(1+x)} = 1+\sum_{n=1}^\infty G_n x^n.\end{equation}
All $G_n$'s are rational numbers alternating in sign ($(-1)^{n-1}G_n>0$ for $n \geq 1$), first several of them being
\[ G_1 = \frac{1}{2}, \quad G_2 = -\frac{1}{12}, \quad G_3 = \frac{1}{24}, \quad G_4 = -\frac{19}{720}, \quad G_5 = \frac{3}{160}, \quad G_6 = -\frac{863}{60480}, \quad \ldots \ .\]

Mascheroni~\cite[pp.21--23]{Mascheroni1790} proved the following beautiful formula for Euler's constant:
\begin{equation}\label{eq:kluyver}  \gamma=\sum_{n=1}^\infty \frac{|G_n|}{n} = \frac12+\frac{1}{24}+\frac{1}{72}+\frac{19}{2880}+\frac{3}{800}+\frac{863}{362880}+\cdots. \end{equation}
This formula was later discovered independently several times, including by Kluyver~\cite{Klu} and N\"orlund~\cite{Norlund1924}. See \cite[p.406]{Blagouchine2016} for more details. Actually, Euler essentially proved this formula as well in \cite[\S7 and \S8]{Euler1789}, but unfortunately there is an erroneous shift by $1$ in each denominator, giving the value $1-\log 2$ instead of $\gamma$.

\begin{definition}  We define another possible analogue $\gaM$ of $\gamma$ in $\A$ by
\[ \gaM\coloneqq\left(\sum_{n=1}^{p-2}\frac{|G_n|}{n}\bmod p\right)_p \in \A. \]
\end{definition}
The reason why we truncate the sum at $p-2$ is that $G_{p-1}$ has $p$ in the denominator (and other $G_n\text{'s}\ (n<p-1)$ do not),
as can be seen from the formula 
\begin{equation}\label{eq:gregbyst1} G_n = \frac{(-1)^n}{n!} \sum_{m=1}^n \frac{(-1)^m \st{n}{m}}{m+1}, \end{equation}
where $\st{n}{m}$ is the Stirling number of the first kind, our convention (unsigned one) being the generating function
identity 
\begin{equation}\label{eq:st1}  \frac{\left(\log(1+x)\right)^m}{m!}=(-1)^m\sum_{n=m}^\infty(-1)^n\st{n}{m}\frac{x^n}{n!}  \end{equation}
holds for each $m\geq 0$. The formula~\eqref{eq:gregbyst1} for $G_n$ can be proved by writing $x$ in the numerator of the defining generating function~\eqref{eq:gregory} of $G_n$'s 
as $e^{\log(1+x)}-1$ and using~\eqref{eq:st1}. 

We mention in passing here the corresponding formula alluded before for Bernoulli numbers in terms of Stirling numbers of the second kind:
\begin{equation}\label{eq:BernStir}
B_n=  (-1)^n \sum_{m=0}^n \frac{(-1)^m m!\sts{n}{m}}{m+1}.
\end{equation}
Here, the Stirling number of the second kind $\sts{n}{m}$ is so defined that the generating series identity
\begin{equation}\label{eq:st2gen} \frac{(e^x-1)^m}{m!}=\sum_{n=m}^\infty \sts{n}{m}\frac{x^n}{n!} \end{equation}
holds for each $m\geq0$ (see for instance~\cite[\S2.1]{AIK}). We use later the following explicit formula (\cite[Proposition~2.6 (6)]{AIK}):
\begin{equation}\label{eq:Stir2expl}
\sts{n}{m}=\frac{(-1)^m}{m!}\sum_{l=0}^m(-1)^l\binom{m}{l}l^n.
\end{equation}

Our first result is a relation between $\gaM$ and $\gaW$. To state it, recall the Fermat quotient
\[ q_p(x) = \frac{x^{p-1} - 1}{p} \in\Z_{(p)}\] 
for any non-zero rational number $x$ whose denominator and numerator are both prime to $p$. We define $\log_\A\colon\Q^{\times}\to \A$ by
\[ \log_\A(x) \coloneqq\left(q_p(x)\bmod p\right)_p\in\A \]
and set 
\[ \ell_\mathcal{A}(x)\coloneqq x\log_\A(x)\in\A. \] 
As can be easily seen (and well known), $\log_\A$ satisfies the functional equation
\begin{equation}\label{eq:loglaw}
\log_\A(xy)=\log_\A(x)+\log_\A(y).
\end{equation}
Silverman~\cite{Silverman1988} proved that $\ker(\log_\A)=\{1,-1\}$ under the ABC-conjecture.

\begin{theorem}\label{thm:main1} We have
	\[
		\gaM = \gaW + \ell_\mathcal{A}(2) - 1.
	\]
\end{theorem} 

\begin{proof}  What we need to show is the congruence 
	\begin{equation}\label{gregwils}
		\sum_{n=1}^{p-2} \frac{|G_n|}{n} \equiv W_p + 2 q_p(2) - 1 \pmod{p}
	\end{equation}
for sufficiently large $p$. Actually, \eqref{gregwils} holds for all $p$ except $p=2$. 
From the defining equation~\eqref{eq:gregory}, we can derive a recurrence for $G_n$, 
\[
	 \sum_{n=1}^{k-1} \frac{|G_n|}{k-n}  = \frac{1}{k}
\]
for $k\geq 2$, and in particular for $k=p>2$, we obtain
\begin{align}\label{Gregory-p-1}
	-G_{p-1} - \frac{1}{p} = -\sum_{n=1}^{p-2} \frac{|G_n|}{p-n} \equiv \sum_{n=1}^{p-2} \frac{|G_n|}{n} \pmod{p}.
\end{align}
On the other hand, the formula~\eqref{eq:gregbyst1} for $n=p-1$ gives
\[
	G_{p-1} = \frac{1}{(p-1)!} \sum_{m=1}^{p-2} \frac{(-1)^m \st{p-1}{m}}{m+1} + \frac{1}{p!}.
\]
Together, noting the congruence $\st{p-1}{m}\equiv1\pmod{p}$ which comes from (see for instance~\cite[\S5]{Hofmodp})
\begin{equation}\label{eq:st1st2}
	\st{n}{m} \equiv \sts{p-m}{p-n} \pmod{p}
\end{equation}
for $1 \leq m \leq n \leq p-1$ and $\sts{n}{1} = 1$ as well as Wilson's theorem, we obtain
\begin{align*}
	\sum_{n=1}^{p-2} \frac{|G_n|}{n} &\equiv -G_{p-1} - \frac{1}{p} =-G_{p-1} +\frac{1}{p!}-\frac{1+(p-1)!}{p(p-1)!}\\
	&\equiv \sum_{m=1}^{p-2} \frac{(-1)^m}{m+1} +W_p\\
	&\equiv  W_p +2q_p(2) - 1\pmod{p}.
\end{align*}
Here, we have used a well-known congruence (which goes back to Eisenstein~\cite{Eisenstein1850})
\[ \sum_{m=0}^{p-2} \frac{(-1)^m}{m+1}\equiv2q_p(2)  \pmod{p}.\qedhere \]
\end{proof}

\section{An interlude, Gregory}

Let us mimic the definition of $Z_\A(k)$ in the introduction for the Gregory coefficient $G_n$.

\begin{definition} For $k\ge2$, define $G_\A(k)\in\A$ by
\[ G_\A(k)\coloneqq\left(G_{p-k}\bmod p\right)_p. \]
\end{definition}

\begin{theorem}\label{thm:gregory}  For $k\ge2$, we have
\[ G_\A(k)= (-1)^k\sum_{j=1}^k(-1)^{j-1}\binom{k}{j}\la(j+1). \]
\end{theorem}
	
\begin{proof} 
Assume that $p\geq k+1$. Starting with the right-hand side, we compute
\begin{align*}
&(-1)^k\sum_{j=1}^k (-1)^{j-1}\binom{k}{j}\frac{(j+1)^p-(j+1)}p\\
&=\frac{(-1)^k}{(k+1)p}\sum_{j=1}^k(-1)^{j-1}\binom{k+1}{j+1}(j+1)^{p+1}-\frac{(-1)^k}{p}\\
&=\frac{(-1)^k}{(k+1)p}\sum_{j=1}^{k+1}(-1)^j\binom{k+1}{j}j^{p+1}\\
&=-k!\cdot\frac1p\sts{p+1}{k+1}.
\end{align*}
Here, we have used $ \sum_{j=1}^k(-1)^{j-1}\binom{k}{j}(j+1)=1\ (k\ge2)$ and~\eqref{eq:Stir2expl}.
Our goal is then to prove the congruence for $p\geq k+1$
\begin{equation}\label{cong:gregory}
G_{p-k}\equiv-k!\cdot\frac{1}{p}\sts{p+1}{k+1}\pmod{p}.
\end{equation}
From the explicit formula~\eqref{eq:gregbyst1} and the congruence~\eqref{eq:st1st2}, we have
\begin{align*}
G_{p-k}&=\frac{(-1)^{k-1}}{(p-k)!}\sum_{m=1}^{p-k}\frac{(-1)^m\st{p-k}{m}}{m+1}\\
&\equiv\frac{(-1)^{k-1}}{(p-k)!}\sum_{m=1}^{p-k}\frac{(-1)^m\sts{p-m}{k}}{m+1}\pmod{p}\\
&=\frac{(-1)^{k-1}}{(p-k)!}\sum_{m=k}^{p-1}\frac{(-1)^{p-m}\sts{m}{k}}{p-m+1}.
\end{align*}
By noting that the congruence 
\[ \frac1p\binom{p}{n}=\frac{(p-1)\cdots(p-n+1)}{n!}\equiv \frac{(-1)^{n-1}}{n}\pmod{p} \]
holds for $1\leq n<p$, we have
\[ G_{p-k}\equiv \frac{(-1)^{k-1}}{(p-k)!}\sum_{m=k}^{p-1}\frac1p\binom{p}{p-m+1}\sts{m}{k}=
\frac{(-1)^{k-1}}{(p-k)!\,p} \sum_{m=1}^{p-1}\binom{p}{m-1}\sts{m}{k}\pmod{p},\]
where $\sts{m}{k}=0$ if $m<k$.
In the following computation, we use a formula for $n\geq k$
\[ \sum_{m=1}^n \binom{n}{m}\sts{m}{k}=\sts{n+1}{k+1}, \]
which is well known (for instance Jordan~\cite[(80)]{Jordan1933}), and can be derived from~\eqref{eq:st2gen} (with $m=k+1$) by differentiating once and use~\eqref{eq:st2gen}
again on the left to compare the coefficients of $x^n$.
Then, we compute by using basic binomial and Stirling identities
\begin{align*} 
\sum_{m=1}^{p-1}\binom{p}{m-1}\sts{m}{k}&=\sum_{m=1}^{p-1}\left(\binom{p+1}{m}-\binom{p}{m}\right)\sts{m}{k}\\
&=\sum_{m=1}^{p-1}\binom{p+1}{m}\sts{m}{k}-\sum_{m=1}^{p-1}\binom{p}{m}\sts{m}{k}\\
&=\sts{p+2}{k+1}-(p+1)\sts{p}{k}-\sts{p+1}{k}-\sts{p+1}{k+1}+\sts{p}{k}\\
&=k\sts{p+1}{k+1}-p\sts{p}{k}.
\end{align*}
Now, since $\sts{p}{k}\equiv0\pmod{p}$ by $p \geq k+1$ and $\frac1{(p-k)!}\equiv(-1)^k(k-1)!\pmod{p}$, we have proved \eqref{cong:gregory} as desired.
\end{proof}

\begin{remark}  There is a nice formula for $\gamma$ as an infinite series involving $\log$'s of natural numbers (Euler~\cite[\S6]{Euler1789} and Ser~\cite{Ser1926}):
\[ \gamma=\sum_{n=1}^\infty \frac1{n+1}\sum_{j=1}^n(-1)^{j-1}\binom{n}{j}\log(j+1). \]
In view of this and Mascheroni's formula~\eqref{eq:kluyver}, our \cref{thm:gregory} just proved
is quite intriguing. And at the same time a natural question emerges:
What is the reason of the appearance of $\la$ instead of $\log_\A$ in our analogous formulas?
\end{remark}

About the non-vanishing of the value $G_\A(k)$, the following can be said: 
\begin{theorem}
For $k\geq 2$, $ G_\A(k)\neq0$ under the ABC-conjecture.
\end{theorem}
\begin{proof}
By \cref{thm:gregory} and \eqref{eq:loglaw}, we have
\[
G_\A(k)=(-1)^k\log_\A\left(\prod_{j=1}^k(j+1)^{(-1)^{j-1}(j+1)\binom{k}{j}}\right).
\]
Since the product inside the parentheses is a positive rational number not equal to $1$, $G_\A(k)$ does not vanish by Silverman's theorem in \cite{Silverman1988} under the ABC-conjecture.
\end{proof}

\section{Yet more variations, Kluyver}

Kluyver~\cite{Klu} also proved an infinite family of formulas for $\gamma$ (or rather, the difference of $\gamma$ and the 
usual quantity whose limit is $\gamma$) in terms of $G_n$'s:
\begin{equation}\label{eq:kluyverm} \gamma=m!\sum_{n=1}^\infty \frac{|G_n|}{n(n+1)\cdots (n+m)}+H_m-\log(m+1),
\end{equation}
where $m$ is a positive integer and $H_m=\sum_{j=1}^m\frac1j$, the $m$th harmonic number.
If we understand $H_0=0$, this includes the formula~\eqref{eq:kluyver} as the case $m=0$.
Binet~\cite[p.136]{Binet} and N\"orlund~\cite[p.244]{Norlund1924} also proved formulas for the digamma function that leads to \eqref{eq:kluyverm}.

Now, let us define for each $m\ge1$ an element $\gaK{m}$ in $\A$ as an analogue of~\eqref{eq:kluyverm} by
\[ \gaK{m}\coloneqq\left(m!\sum_{n=1}^{p-m-1}\frac{|G_n|}{n(n+1)\cdots (n+m)} \bmod p\right)_p+H_m-\la(m+1). \]

\begin{remark} Note the discrepancy in that the upper limit in the case of $m=0$ is different from the one in the definition of $\gaM$.
We are curious to uncover the reason (if any) why.
\end{remark}

\begin{theorem}\label{thm:kluyvergeneral} We have for each $m\ge1$
\[ \gaK{m} =\gaW-1+(H_m-1)\la(m+1)+\sum_{j=1}^{m-1}(-1)^{m-j}\binom{m}{j}\frac{\la(j+1)}{m-j}. \]
\end{theorem}
The particular case $m=1$ of this is 
\[\gaK{1}=\gaW-1 \] 
or 
\[ \gaW=\left(\sum_{n=1}^{p-2}\frac{|G_n|}{n(n+1)} \bmod p\right)_p+2-\la(2), \]
which is in contrast to Kluyver's ($m=1$ of~\eqref{eq:kluyverm})
\[ \gamma=\sum_{n=1}^\infty \frac{|G_n|}{n(n+1)}+1-\log2. \]

We prove the theorem by combining the statements in the following lemma. Recall $m$ is a positive integer.

\begin{lemma}
\begin{enumerate}[$1)$]
\item\label{it:1} For $l \geq m+1$, we have an identity
\begin{equation}\label{eq:lem1}
\begin{split}
&m! \sum_{n=1}^{l-m} \frac{|G_n|}{(l-n-m+1) \cdots (l-n) (l-n+1)}\\
&= \frac{m!}{(l-m+1)(l-m+2) \cdots (l+1)}+ \sum_{k=1}^{m} (-1)^{m+k} \binom{m}{k}(H_m-H_{m-k})|G_{l-k+1}|.
\end{split}
\end{equation}
\item\label{it:2} For any prime $p\geq m+1$, we have a congruence
\begin{equation}\label{eq:lem2}
\frac{(-1)^m}{p\binom{p-1}{m}}-\frac1p\equiv H_m \pmod{p}.
\end{equation}

\item\label{it:3} For any prime $p\geq 3$, we have a congruence
\begin{equation}\label{eq:lem3}
-G_{p-1} - \frac{1}{p} \equiv W_p + 2q_p(2) - 1 \pmod{p}. 
\end{equation}

\item\label{it:4} For any prime $p \geq m+1$, we have 
\begin{equation}\label{eq:lem4}
\begin{split}
&\sum_{k=2}^{m} (-1)^k \binom{m}{k} (H_m-H_{m-k})|G_{p-k}| \\
& \equiv  -\sum_{j=1}^{m-1} (-1)^{m-j} \binom{m}{j}\frac{(j+1)q_p(j+1)}{m-j}-H_m\cdot(m+1)q_p(m+1)+2q_p(2) \pmod{p}. 
\end{split}
\end{equation}
When $m+1=p$, $(m+1)q_p(m+1)$ should be read as $p^{p-1} - 1$.
\end{enumerate}
\end{lemma}

\begin{proof} 

\ref{it:1}) We start with the binomial expansion for a non-zero real number $x$ with $|x|<1$
\[ \frac{1-(1-x)^s}{x} = \sum_{k=1}^\infty (-1)^{k-1}\binom{s}{k} x^{k-1}. \]
Here, $s$ is a variable, say in $\R_{>0}$. We differentiate this with respect to $s$ and take a limit $s\to m$.
Let $\psi(s)=\Gamma'(s)/\Gamma(s)$ be the digamma function.
Then, noting the formula
\[ \frac{d}{ds}\binom{s}{k}=\binom{s}{k}\left(\psi(s+1)-\psi(s+1-k)\right), \]
and 
\begin{equation}\label{eq:binderiv} \lim_{s\to m} \binom{s}{k}\left(\psi(s+1)-\psi(s+1-k)\right) =
	\begin{cases} {\displaystyle{\binom{m}{k}}}(H_m-H_{m-k}) & \text{if }m\geq k,  \\[3mm]
		\dfrac{(-1)^{m+k-1}m!}{(k-m)\cdots (k-1)k} & \text{if }1\leq m<k,
	\end{cases}
\end{equation}
(recall $\psi(m+1)=H_m-\gamma$), we have
\begin{align*}
&\frac{\bigl(-\log(1-x)\bigr)\cdot(1-x)^m}x \\
&=\sum_{k=1}^m (-1)^{k-1}\binom{m}{k}(H_m-H_{m-k}) x^{k-1}+\sum_{k=m+1}^\infty \frac{(-1)^{m}m!}{(k-m)\cdots (k-1)k} x^{k-1}.
\end{align*}
If we multiply both sides of this by 
\[ \frac{x}{-\log(1-x)}=1-\sum_{n=1}^\infty |G_n| x^n\]
and compare the coefficients of $x^l$ for $l\geq m+1$, we obtain~\eqref{eq:lem1}.

\ref{it:2})  The congruence~\eqref{eq:lem2} is an easy consequence of 
\[
(-1)^m\binom{p-1}{m}=\prod_{j=1}^m\left(1-\frac{p}{j}\right)\equiv 1-H_mp\pmod{p^2}.
\]

\ref{it:3}) This is a combination of~\eqref{Gregory-p-1} and~\eqref{gregwils}.

\ref{it:4}) Since the case $m=1$ is trivial, we may assume that $m\geq 2$. To prove~\eqref{eq:lem4}, we use the congruence for each $k$ satisfying $2\leq k\leq p-1$
\[
G_{p-k}\equiv(-1)^k\sum_{j=1}^k(-1)^{j-1}\binom{k}{j}(j+1)q_p(j+1)\pmod{p},
\]
which has been proved in the proof of \cref{thm:gregory} and compute
\begin{align*}
&\sum_{k=2}^{m} (-1)^k \binom{m}{k} (H_m-H_{m-k})|G_{p-k}|\\
&\equiv \sum_{k=2}^{m} (-1)^k \binom{m}{k} (H_m-H_{m-k})\,  \sum_{j=1}^k(-1)^{j-1}\binom{k}{j}(j+1)q_p(j+1) \pmod{p}\\
&= \sum_{j=1}^m(-1)^{j-1}(j+1)q_p(j+1)\sum_{k=j}^{m} (-1)^k \binom{k}{j} \binom{m}{k} (H_m-H_{m-k})+2q_p(2).
\end{align*}
We show the identity
\begin{equation}\label{eq:claim} \sum_{k=j}^{m} (-1)^k \binom{k}{j} \binom{m}{k} (H_m-H_{m-k})=
\begin{cases} (-1)^m{\displaystyle \binom{m}{j}\frac{1}{m-j}}& \text{if }0\leq j<m,\\[3mm]
(-1)^mH_m & \text{if }j=m,
\end{cases} \end{equation}
and then we are done. For this, we use 
\begin{equation}\label{eq:binpol} \sum_{k=j}^m(-1)^k \binom{k}{j}\binom{x}{k}=(-1)^m\binom{x}{j}\binom{x-j-1}{m-j},\end{equation}
which is valid for $0\leq j\leq m$.
This follows from $\binom{k}{j}\binom{x}{k}=\binom{x}{j}\binom{x-j}{k-j}$ and
\[ \sum_{k=j}^m(-1)^k\binom{x-j}{k-j}=(-1)^m\binom{x-j-1}{m-j}, \]
which is shown by telescoping.
We differentiate~\eqref{eq:binpol} with respect to $x$ and set $x=m$, and use~\eqref{eq:binderiv} to obtain~\eqref{eq:claim}.  
\end{proof}

\begin{proof}[Proof of $\cref{thm:kluyvergeneral}$] 
Set $l=p-1$ in~\eqref{eq:lem1} and multiply both sides by $(-1)^{m+1}$.
Then we have 	
	\begin{align*}
		&m! \sum_{n=1}^{p-m-1} \frac{|G_n|}{(n+m-p) \cdots  (n+1-p)(n-p)}\\
		&= \frac{(-1)^{m+1}m!}{(p-m)(p-m+1) \cdots p} -\sum_{k=1}^{m} (-1)^{k} \binom{m}{k} (H_m-H_{m-k})|G_{p-k}|\\
		&=\frac{(-1)^{m+1}}{p\binom{p-1}m}-G_{p-1}-\sum_{k=2}^{m} (-1)^{k} \binom{m}{k} (H_m-H_{m-k})|G_{p-k}|\\
		&=-\left(\frac{(-1)^m}{p\binom{p-1}m}-\frac1p\right)-G_{p-1}-\frac1p-\sum_{k=2}^{m} (-1)^{k} \binom{m}{k} (H_m-H_{m-k})|G_{p-k}|.
	\end{align*}
Reducing this modulo $p$ and using~\eqref{eq:lem2}, \eqref{eq:lem3} and~\eqref{eq:lem4}, we obtain 
\begin{align*}
&m!\sum_{n=1}^{p-m-1}\frac{|G_n|}{n(n+1)\cdots (n+m)}+H_m-(m+1)q_p(m+1)\\
&\equiv W_p - 1+(H_m-1)\cdot(m+1)q_p(m+1)+\sum_{j=1}^{m-1} (-1)^{m-j} \binom{m}{j}\frac{(j+1)q_p(j+1)}{m-j} \pmod{p},
\end{align*}
which is the desired congruence.	
\end{proof}
\begin{remark}
Our main results (\cref{thm:main1} and \cref{thm:kluyvergeneral}) say that our various analogues of $\gamma$ in $\A$ differ only by linear combinations of $\log_{\A}(j)$ ($j\in \mathbb{N}$) and 1.
This may suggest something very interesting in light of a recent work of Rosen \cite{Rosen} on ``periods'' in $\A$ (or more generally in the ring $\prod_p\mathbb{Q}_p/\bigoplus_p\mathbb{Q}_p$).
According to him, $\log_{\A}(j)$'s are one of the most simple periods in $\A$ (like $\log x = \int_1^x\frac{dt}{t}$ in $\mathbb{R}$).
On the other hand, Euler's constant $\gamma$ (in $\mathbb{R}$) is believed to be \emph{not} a period in the sense of Kontsevich--Zagier \cite{KontsevichZagier2001}.
If we believe the parallelism holds in $\A$, an analogue $\gamma_{\A}$ of $\gamma$ in $\A$ might arguably not a period.
Our results then suggest that in whatever way we define an analogue $\gamma_{\A}$, it might possibly give a unique class of elements in $\A$ modulo a linear combination of $\log_{\A}(j)$'s and 1.
\end{remark}
We anyway believe that there are more and more interesting arithmetic in the ring $\A$.

\subsection*{Acknowledgements}
This work was supported by JSPS KAKENHI Grant Numbers JP21H04430, JP21K18141 (Kaneko), JP20K14292 (Matsusaka), and JP21K13762 (Seki).

\leavevmode\\
Masanobu Kaneko\\
Faculty of Mathematics, Kyushu University\\
Motooka 744, Nishi-ku, Fukuoka, 819-0395, JAPAN\\
e-mail: mkaneko@math.kyushu-u.ac.jp\\
\\
Toshiki Matsusaka\\
Faculty of Mathematics, Kyushu University\\
Motooka 744, Nishi-ku, Fukuoka, 819-0395, JAPAN\\
e-mail: matsusaka@math.kyushu-u.ac.jp\\
\\
Shin-ichiro Seki\\
Department of Mathematical Sciences, Aoyama Gakuin University\\
Fuchinobe 5-10-1, Chuo-ku, Sagamihara, Kanagawa, 252-5258, JAPAN\\
e-mail: seki@math.aoyama.ac.jp

\end{document}